\documentclass[11pt]{amsart}

\usepackage{enumerate}

\usepackage[utf8]{inputenc}
\usepackage[T1]{fontenc}
\usepackage[mathscr]{eucal}

\usepackage[a4paper, twoside=false, vmargin={2cm,3cm}, includehead]{geometry}

\newtheorem{lemma}{Lemma}[section]
\newtheorem{theorem}[lemma]{Theorem}
\newtheorem*{theorem*}{Theorem}

\newtheorem*{problem*}{Problem}

\newtheorem*{convention}{Convention}

\newtheorem*{definition}{Definition}
\newtheorem*{remark}{Remark}
\newtheorem*{remarks}{Remarks}

\newcommand{\C}{{\mathbb C}}

\newcommand{\N}{{\mathbb N}}
\renewcommand{\P}{{\mathbb P}}
\newcommand{\Q}{{\mathbb Q}}
\newcommand{\R}{{\mathbb R}}
\newcommand{\T}{{\mathbb T}}
\newcommand{\Z}{{\mathbb Z}}

\newcommand{\CM}{{\mathcal M}}

\newcommand{\CX}{{\mathcal X}}

\newcommand{\bm}{{\mathbf{m}}}

\newcommand{\bk}{{\mathbf{k}}}

\newcommand{\one}{\mathbf{1}}

\newcommand{\ve}{\varepsilon}
\newcommand{\wh}{\widehat}
\newcommand{\wt}{\widetilde}
\newcommand{\e}{\mathrm{e}}

\newcommand{\tN}{{\widetilde N}}
\newcommand{\inv}{^{-1}}
\newcommand{\st}{{\text{\rm st}}}
\newcommand{\un}{{\text{\rm un}}}
\DeclareMathOperator{\id}{id}
\DeclareMathOperator{\real}{Re}
\renewcommand{\Re}{\real}
\DeclareMathOperator{\Imag}{Im}
\renewcommand{\Im}{\Imag}

\newcommand{\norm}[1]{\lVert #1 \rVert}

\begin{document}

\title{Asymptotics for   multilinear averages of multiplicative functions}

\author{Nikos Frantzikinakis}
\address[Nikos Frantzikinakis]{University of Crete, Department of mathematics, Voutes University Campus, Heraklion 71003, Greece} \email{frantzikinakis@gmail.com}
\author{Bernard Host}
\address[Bernard Host]{
Universit\'e Paris-Est Marne-la-Vall\'ee, Laboratoire d'analyse et
de math\'ematiques appliqu\'ees, UMR CNRS 8050, 5 Bd Descartes,
77454 Marne la Vall\'ee Cedex, France }
\email{bernard.host@u-pem.fr}

\thanks{B. Host was partially supported by Centro de Modelamiento Matem\'atico, Universitad de Chile.}

\begin{abstract}
 A celebrated result of Hal\'asz describes the asymptotic behavior of the arithmetic  mean   of an arbitrary  multiplicative function
 with values on the unit disc. We extend this result to multilinear averages of  multiplicative functions  providing similar asymptotics, thus verifying a two dimensional variant of a conjecture of Elliott.  As a consequence, we get several convergence results for such  multilinear expressions, one of which generalizes a well known convergence result of Wirsing.
 The key ingredients are a recent structural result for  multiplicative functions with values on the unit disc proved by the authors
 and the  mean value theorem of Hal\'asz.
\end{abstract}



\maketitle

\section{Background and main results}
The problem of existence of the mean value
\begin{equation}\label{E:MV}
M(f):=\lim_{N\to\infty}\frac{1}{N}\sum_{n=1}^Nf(n)
\end{equation}
of a   multiplicative function $f$ with values on the unit disc
has been an object of intense study in analytic number theory.  P.~Erd\"os and A.~Wintner~\cite{E57} conjectured that if $f$ takes real values and is bounded by $1$, then the mean value $M(f)$ exists.   H.~Delange~\cite{D61} verified this when  $\sum_{p\in \P} \frac{1-f(p)}{p}<\infty$ in which case  $M(f)\neq 0$ unless $f(2^k)=-1$ for every $k\in \N$.
 The proof of the Erd\"os-Wintner conjecture was completed by E.~Wirsing~\cite{Wi67};  building on earlier work of his \cite{Wi61, Wi64}, he  showed that if the previous series diverges, then $M(f)=0$. Note that the prime number theorem corresponds to the very special case of this result where $f$ is the M\"obius function. The previous results were extended to complex valued multiplicative functions   by a  celebrated result of G.~Hal\'asz~\cite{Hal68}.
In order to give  the precise statement we need a definition:
\begin{definition}[Slowly-varying sequences]
We say that $w\colon \N \to \R$ is \emph{slowly-varying} if
$\max_{x\leq n\leq x^2}|w(n)-w(x)|\to 0$ as $x\to \infty$.
\end{definition}

 The next result is due to Hal\'asz~\cite{Hal68}.
We use it in the form given in~\cite[Theorem 6.2]{E79} and \cite[Chapter III.4]{T}.

\begin{theorem*}[Hal\'asz mean value theorem]
\label{T:Halasz}
Let $f$ be a  multiplicative function that takes values on the unit disc. Then there exist  constants $c\in \C$, $t\in \R$, and a slowly-varying sequence
  $w\colon \N\to \R$ such that
$$
  \frac{1}{N}\sum_{n=1}^Nf(n)=c N^{it}\e(w(N))+o_{N\to \infty}(1).
 $$
 If $c\neq 0$, then  $t$ is the unique  real number such that $\sum_{p\in \P}\frac{1}{p}\,\bigl(1-\Re\bigl(f(p) p^{-it}\bigr)\bigr)<\infty$ ($c=0$ if no such number exists),
 in which case we
   can take
$w(N)=\frac 1{2\pi} \sum_{p\in \P, p\leq N}\frac{1}{p}\Im(f(p)p^{-it})$.
 \end{theorem*}
 \begin{remarks}
$\bullet$   Explicit quantitative bounds exist in the case $c=0$; see for example~\cite[Chapter III.4, Corollary 6.3]{T}.

$\bullet$ If  $f(n)=n^{it}$ for some $t\neq 0$, then  $ \frac{1}{N}\sum_{n=1}^Nf(n)=\frac{ N^{it}}{1+it}+o_{N\to \infty}(1)$; hence $M(f)$ does not exist. Lending terminology from \cite{GS15}, the theorem of Hal\'asz implies that the mean value $M(f)$ exists and is $0$ unless $f(n)$ ``pretends'' to be $n^{it}$ for some $t\in \R$.
 \end{remarks}

Our main goal is to extend the convergence result of Wirsing and the asymptotic formula of Hal\'asz to multilinear averages of multiplicative functions.
 Although the one dimensional multilinear averages
$\frac{1}{N}\sum_{n=1}^N  \prod_{j=1}^\ell f_j(k_jn+ a_j),$ where $k_j,a_j\in \N$ and $f_j$ are  multiplicative functions with modulus at most $1$,
are notoriously hard to analyze,
a conjecture of P.~Elliott~\cite[Conjecture~I]{El90,El94}
predicts that they satisfy asymptotics similar to those in Hal\'asz's theorem. Our main results verify these asymptotics when the one dimensional affine-linear forms $k_jn+ a_j$, $j=1,\ldots, \ell$,  are replaced with higher dimensional linear forms, or affine-linear forms with
 pairwise independent linear parts.
 The setup is as follows: We are given complex valued
multiplicative functions $f_1,\ldots, f_\ell$ and linear forms $L_1,\ldots, L_\ell\colon \N^d\to \N$ given  by
$$
L_j(\bm)=\bk_j\cdot \bm\text{ for some }\bk_j\in \N^d\ .
$$
\begin{convention}
Henceforth, we  assume that all multiplicative functions  take values on the  interval $[-1,1]$ or the  unit disc  depending on whether they are real or complex valued. Furthermore, with $[N]$ we denote the set $\{1,\ldots, N\}$.
\end{convention}
We are interested in studying the asymptotic behavior of
the averages
\begin{equation}\label{E:multi1}
\frac{1}{N^d}\sum_{\bm \in [N]^d} \prod_{j=1}^\ell f_j(L_j(\bm)).
 \end{equation}
We remark that for the purposes of this article the special case  where $d=2$  and
$L_j(m,n)=m+(j-1)n$, $j=1,\ldots,\ell$, is essentially as hard as the general case.
If the linear forms are pairwise (linearly) independent, that is, no two are rational multiples of each  other,  and one of the multiplicative functions is the M\"obius or the Liouville function, then
 B.~Green and  T.~Tao showed in \cite[Proposition~9.1]{GT10}, modulo conjectures which  were subsequently verified in \cite{GT12a, GT12b, GTZ12c}, that the averages \eqref{E:multi1} converge to $0$.

 Our first result is the following:
\begin{theorem}[Asymptotic form of multilinear averages]\label{T:Main}
 Let $d\in \N$, $f_1,\ldots, f_\ell$  be  complex valued  multiplicative functions with modulus at most $1$, and  $L_1,\ldots, L_\ell\colon \N^d\to \N$ be linear forms. Then there exist  $c\in \C$,  $t\in \R$, and a slowly-varying sequence $w\colon \N\to \R$, such that
\begin{equation}\label{E:multi2}
\frac{1}{N^d}\sum_{\bm \in [N]^d} \prod_{j=1}^\ell f_j(L_j(\bm))=cN^{it}\e(w(N))+o_{N\to\infty}(1).
\end{equation}
If in addition we assume that the linear forms are pairwise independent\footnote{We can always reduce to this case
after putting together multiplicative functions evaluated at linear forms that are pairwise dependent.}, then  $t=\sum_{j=1}^\ell t_{f_j}$ and we can take $w=\sum_{j=1}^\ell w_{f_j}$,
where for $j=1,\ldots, \ell$,
 $t_{f_j}$ are the real numbers and $w_{f_j}$ are the slowly-varying sequences defined in Theorem~\ref{T:HalaszAP} below. Furthermore, the
 constant  $c$ in \eqref{E:multi2}
is $0$ unless all $f_j$ are ``pretentious'', meaning, for $j=1,\ldots, \ell$
there exist $t_j\in \R$ and
   Dirichlet characters $\chi_j$ such that  $\sum_{p\in \P}\frac{1}{p}\,\bigl(1-\Re\bigl(\chi_j(p)f_j(p) p^{-it_j}\bigr)\bigr)<\infty$.
 \end{theorem}

\begin{remarks}
$\bullet$ In a recent preprint   K.~Matom\"aki, M.~Radziwi{\l}{\l},   T.~Tao~\cite{MRT15}, using techniques from \cite{MR15},  prove that the   averages   $\frac{1}{M^\ell}\sum_{m_1,\ldots, m_\ell\in [M]}|\frac{1}{N}\sum_{n\in [N]}\prod_{j=1}^\ell f_j(n+m_j)|$
converge to $0$ if  $M=M(N)$ increases  to infinity with $N$ at an arbitrary speed and at least one of the multiplicative functions is ``non-pretentious'' in a certain uniform sense. Such  results are complementary to ours and
rely on very  different techniques.\footnote{In \cite{MRT15} the authors treat  averages taken over a short interval $[M]$ and a long interval $[N]$; key to their analysis is a   correlation estimate over short intervals between   multiplicative functions and linear complex exponential sequences.
In  \cite{FH14} and in this article, the main difficulty is different,  we take $M=N$ but consider averages
taken over   arbitrary  subspaces of $\Z^\ell$ which are given in parametric form;  at the heart of this analysis   lies a  correlation estimate of  multiplicative functions with nilsequences.}

$\bullet$
   If the linear forms are pairwise independent, then Theorem~\ref{T:Main} and all subsequent results remain valid with  the affine-linear forms  $L_j(\bm)+a_j$, $a_j\in \Z$,   in place of  the linear forms $L_j(\bm)$ for $j=1,\ldots, \ell$.
   Interestingly, this is no longer true if two of the linear forms are dependent, even if the affine-linear forms are independent;  an example
    from \cite[Appendix~A]{MRT15}
  shows that there exists a non-pretentious multiplicative function $f$ such that the averages  $ \frac{1}{N}\sum_{n=1}^Nf(n)\overline{f(n+1)}$ do not converge to zero.
\end{remarks}

For  $\ell= 3$,  the asymptotic formula \eqref{E:multi2} can be proved by combining tools from discrete Fourier analysis, a result of H.~Daboussi~\cite{D74, DD74} that provides estimates  for the Fourier coefficients of a multiplicative function, and the previously mentioned asymptotic formula of  Hal\'asz. For $\ell\geq 4$
 classical discrete Fourier analysis tools turn out to be insufficient for the task at hand, the reason being that for general bounded sequences the modulus of the averages \eqref{E:multi1} is not controlled by the maximum modulus of the Fourier coefficients of the individual functions (thought of as functions of $\Z_N$). To overcome this obstacle we use a deep structural result from \cite{FH14} (see Theorem~\ref{T:Structure}) which was proved using the toolbox of ``higher order Fourier analysis''. For our particular needs it implies that the
 general  multiplicative function with modulus at most $1$ can be decomposed in two terms, one that is approximately periodic and another that contributes negligibly to the   averages \eqref{E:multi1}.
 The contribution  of the structured component is then analyzed  using  an extension of Hal\'asz's asymptotic formula (see Theorem~\ref{T:HalaszAP}) and after some effort the outcome is the asymptotic  formula \eqref{E:multi2}. The details are given in Section~\ref{S:proofs}.

    Using Theorem~\ref{T:Main} we deduce the following generalization of the convergence result of E.~Wirsing which deals with  multilinear averages of real valued  multiplicative functions:
\begin{theorem}[Wirsing's theorem for multilinear averages] \label{T:1}
Let $f_1,\ldots, f_\ell$ be real valued  multiplicative functions with modulus at most $1$
 and $L_1,\ldots, L_\ell\colon \N^d\to \N$ be linear forms. Then the averages  \eqref{E:multi1}
converge as $N \to \infty$.
 \end{theorem}
 \begin{remark}
If the linear forms are pairwise independent, then  the result remains true if   we  replace the assumption that the $f_j$'s take  real values with the assumption that the
mean values $\lim_{N\to \infty} \frac{1}{N}\sum_{n=1}^N f_j(n)\chi(n)$ exist
 for every Dirichlet character $\chi$.\footnote{
 Whether this holds can be  verified using the following consequence of the mean value theorem of Hal\'asz
   (see~\cite[Theorem~6.3]{E79}): The mean value $\lim_{N\to \infty} \frac{1}{N}\sum_{n=1}^N f(n)$ exists if and only if  either $(i)$ $\sum_{p\in \P}  p^{-1}(1-\Re(f(p)p^{-it}))=\infty$ for every $t\in \R$, or $(ii)$ $\sum_{p\in \P}p^{-1}(1-f(p))$ converges, or $(iii)$ for some $t\in \R$ we have  $\sum_{p\in \P}  p^{-1}(1-\Re(f(p)p^{-it}))<\infty$  and $f(2^k)=-2^{ikt}$ for all $k\in \N$.}
 \end{remark}

It follows from the asymptotic formula given in Hal\'asz's theorem that if  $f$ is a
complex valued  multiplicative function with modulus at most $1$, then the
averages $\Big|\frac{1}{N}\sum_{n=1}^Nf(n)\Big|$ converge as $N\to \infty$. A similar result extends to multilinear averages:
\begin{theorem}[Convergence of the modulus]
\label{T:2}
 Let $f_1,\ldots, f_\ell$
  be complex valued  multiplicative functions with modulus at most $1$,
 and $L_1,\ldots, L_\ell\colon \N^d\to \N$ be linear forms. Then the modulus of the  averages \eqref{E:multi1}
converges as $N \to \infty$.
 \end{theorem}

 Interestingly, although  for  complex valued  multiplicative functions  with modulus at most $1$ the multilinear averages \eqref{E:multi1} do not in general converge, we do have convergence if in \eqref{E:multi1} each multiplicative function is paired up with its complex conjugate.
\begin{theorem}
\label{T:3}
 Let $f_1,\ldots, f_\ell$ be  complex valued  multiplicative functions with modulus at most $1$ and
 $L_j,L_j' \colon\N^d\to \N$, $j=1,\ldots, \ell$,   be pairwise independent linear forms.
  Then   the averages
$$
\frac{1}{N^d}\sum_{\bm\in [N]^d} \prod_{j=1}^\ell f_j(L_j(\bm))\cdot \overline{f_j}(L'_j(\bm))
$$
converge as $N \to \infty$.
 \end{theorem}
We give an application of Theorem~\ref{T:3} in ergodic theory.

\begin{theorem}
\label{T:4}
 Let  $(X,\CX,\mu)$ be a probability space and  for $n\in \N$ let $T_n\colon X\to X$ be
 invertible measure preserving
transformations  that satisfy
$ T_1:=\id$ and $T_m\circ T_n=T_{mn}$ for every $m,n\in\N$.
  Let also
 $L_j,L_j'\colon  \N^d\to \N$, $j=1,\ldots, \ell$,   be pairwise independent linear forms. Then for all $F,G\in L^2(\mu)$  the averages
\begin{equation}
\label{E:ergodic}
\frac{1}{N^d}\sum_{\bm\in [N]^d}
 \int F(T_{\prod_{j=1}^\ell L_j(\bm)}x)\cdot G(T_{\prod_{j=1}^\ell L'_j(\bm)}x)\ d\mu
\end{equation}
converge   as $N \to \infty$.
\end{theorem}

\begin{remark}
The averages \eqref{E:ergodic} were studied in \cite{FH14}  when $d=\ell=2$ and $F=G=\one_A$ with $\mu(A)>0$ in order to prove partition regularity
for certain quadratic equations.
\end{remark}

Note that Theorem~\ref{T:4} is non-trivial even for $d=2$ and  $\ell=1$. Interestingly,  the averages $
\frac{1}{N}\sum_{n=1}^N
 \int F(T_{n}x)\cdot G(x)\ d\mu
$ do not necessarily converge. To see this, let $T_nx=x +\log n  \bmod{1}$ act on $\T$ with the Haar measure and take $F(x)=\e(x)$, $G(x)=\e(-x)$. Then the ergodic averages take the form
$\frac{1}{N}\sum_{n=1}^Nn^{2\pi i}$. Hence, they do not converge.

It is  natural to ask whether  we have convergence of the mean value of a  multiplicative function evaluated at
homogeneous polynomials that do not  necessarily factor linearly.
\begin{problem*}
 Let $f$ be a real valued bounded completely multiplicative function and $P\in\Z[x,y]$ be a homogeneous polynomial with values on the positive integers. Do the averages
 $$
 \frac{1}{N^2}\sum_{1\leq m, n\leq N}f(P(m,n))
 $$
converge as $N\to \infty$? If the multiplicative function $f$ takes complex values, does the modulus of the above averages  converge as $N\to\infty$?
\end{problem*}

  Using Theorem~\ref{T:1} we get  a positive answer when $P(m,n)=\prod_{j=1}^\ell L_j(m,n)$ where $L_j\colon\N^2\to \N$, $j=1,\ldots, \ell$,  are linear forms.
Lastly, we remark  that  for $s\geq 2$ the Gowers norms $\norm{f_N}_{U^s(\Z_N)}$ (defined in Section~\ref{S:2}) of a  multiplicative function $f$ with modulus at most $1$ do not necessarily converge as $N\to \infty$ even if $f$ takes  values in the set $\{-1,1\}$.
To see this, consider the (non-completely) multiplicative function defined by $f(n)=(-1)^{n+1}$.
Then $\lim_{N\to \infty}\norm{f_{2N}}_{U^2(\Z_{2N})}=1$ and $\lim_{N\to \infty}\norm{f_{2N+1}}_{U^2(\Z_{2N+1})}=1/2$. On the other hand, Theorem~\ref{T:3} implies that if $f$ is  a complex valued  multiplicative function with modulus at most $1$, then  the averages
$$
\frac{1}{N^3}\sum_{1\leq m,n,r\leq N} f(m) \, \overline{f}(m+n) \, \overline{f}(m+r)\,  f(m+n+r)
$$
converge as $N\to \infty$. Note that taking the previous average over $\Z_N$ leads to $\norm{f_N}_{U^2(\Z_N)}^4$.
  Theorem~\ref{T:3} also implies convergence  for higher dimensional variants of such averages.

\subsection{Notation and conventions}
  We denote by $\N$ the set of positive integers and by $\P$  the set of prime numbers.
 For  $d,N\in\N$  we let $\Z_N:=\Z/N\Z$, $[N]:=\{1,\dots,N\}$, $[N]^d=[N]\times\cdots \times[N]$. We  identify $[N]$ and $\Z_N$ in the obvious way.  We let  $\e(t):=e^{2\pi it}$.
A linear form  $L\colon \N^d\to \N$ is a function  of the form  $L(\bm)=\bk\cdot\bm$
 for some $\bk\in \N^d$.
Two linear forms are independent if one is not a rational multiple of the other.
With $o_{N\to \infty}(1)$ we denote a quantity that converges to $0$ when  $N\to \infty$ and all other implicit parameters are fixed.
\begin{definition}
A function $f\colon \N\to \C$ is \emph{multiplicative} if $f(mn)=f(m)f(n)$ whenever  $(m,n)=1$. Moreover,
$f$ is \emph{completely multiplicative} if this relation holds for all $m,n\in \N$.

With $\CM_\R$, $\CM_\C$,  $\CM_\T$,
we denote the set  of  multiplicative functions on $\N$  that take values on $[-1,1]$, the unit disc, and the unit circle
 correspondingly.

 A Dirichlet character (denoted by $\chi$) is a  completely multiplicative  function  that is periodic and not identically $0$.
\end{definition}
 \section{Two key ingredients}\label{S:2}
To prove Theorem~\ref{T:Main} we will use a structural result for multiplicative functions proved by the authors in \cite{FH14} and the following  extension of the mean value theorem of Hal\'asz
which  can be found in the exact form stated here in \cite[Theorem 1]{D83}:
\begin{theorem}[Hal\'asz-Delange~\cite{D72,D83, Hal68}]\label{T:HalaszAP}
  Let $f\in \CM_\C$. Then there exists a constant  $t\in \R$ and a slowly-varying sequence
  $w\colon \N\to \R$, such that  the following holds: For every $a,b\in \N$
  there exists a constant
   $c=c_{f,a,b}\in \C$ such that
\begin{equation}
\label{E:fab}
  \frac{1}{N}\sum_{n=1}^Nf(an+b)=c N^{it}\e(w(N))+o_{N\to \infty}(1).
\end{equation}
 If $c_{f,a,b}= 0$ for all $a,b\in \N$,  we set $t_f:=0$ and $w_f:=0$. If $c_{f,a,b}\neq 0$ for some $a,b\in \N$, then  $t=t_f$ is the unique real number
  for which there exists a  primitive Dirichlet character  $\chi=\chi_f$ such that  $\sum_{p\in \P}\frac{1}{p}\,\bigl(1-\Re\bigl(\chi(p)f(p) p^{it}\bigr)\bigr)<\infty$. Furthermore,
 $\chi$  is uniquely determined and we set
$w_f(N):=\sum_{p\in \P, p\leq N}\frac{1}{p}\Im(\chi(p)f(p)p^{-it})$.
 \end{theorem}

\begin{remarks}
$\bullet$ It is important that neither $t_f$ nor $w_f$ depend on $a$ or $b$.

$\bullet$ It follows  from \eqref{E:fab} and the definition of $w_f$, that for $f\in \CM_\C$  we have $t_{\bar{f}}=-t_f$ and   $w_{\bar{f}}=-w_f$.
 Hence, for $f\in \CM_\R $ we have $t_f=0$ and  $w_f=0$. In this case,   the averages in \eqref{E:fab} converge for all $a,b\in \N$; explicit formulas for the limit appear in \cite{BGS12, D83}.
\end{remarks}

We now turn to the structural result; in order to state it  we need to introduce  some notation from~\cite{FH14}.
Given $f\colon \N\to \C$ and $N\in \N$ we  let
 $$
f_N:=f\cdot \one_{[N]}
 $$
 and whenever appropriate, identifying the interval $[N]$ with $\Z_N$,  we consider $f_N$ as a function in $\Z_N$.
By a \emph{kernel} on $\Z_N$ we mean a non-negative function on $\Z_{N}$ with average  $1$.
For every prime number $N$ and $\theta>0$,  in  \cite[Section 3.3]{FH14}  we  defined two positive integers $Q=Q(\theta)$ and $V=V(\theta)$, and for  $N>2QV$, a kernel   $\phi_{N,\theta}$ was defined as follows:
The spectrum  of $\phi_{N,\theta}$ is the set
\begin{equation}
\label{eq:def-Xi}
\Xi_{N,\theta}:=\Big\{\xi\in \Z_N\colon
\Bigl\Vert\frac{Q \xi}N\Bigr\Vert<
\frac{QV}N\Big\},
\end{equation}
and
\begin{equation}
\label{eq:fourier-phi} \widehat{\phi_{N,\theta}}(\xi) :=\begin{cases}
\displaystyle 1-\Bigl\Vert \frac{Q\xi}N\Bigr\Vert\, \frac
N{ QV}
&\ \  \text{if }\  \xi\in\Xi_{N,\theta} \  ;\\
0 & \ \ \text{otherwise.}
\end{cases}
\end{equation}
We recall the definition of the    $U^s$-Gowers uniformity
norms from \cite{G01}.
\begin{definition}[Gowers norms on a cyclic group~\cite{G01}]
Let $N\in \N$  and $a\colon \Z_N\to \C$. For $s\in \N$ the \emph{Gowers $U^s(\Z_N)$-norm} $\norm a_{U^s(\Z_N)}$ of $a$ is defined inductively as follows:   For every $t\in\Z_N$ we write $a_t(n):=a(n+t)$. We let
$$
\norm a_{U^1(\Z_N)}:=\Big|\frac{1}{N}\sum_{n\in \Z_N} a(n)\Big|
$$
and for every $s\in \N$ we let
$$
\norm a_{U^{s+1}(\Z_N)}:=\Bigl(\frac{1}{N}\sum_{t\in \Z_N}\norm{a\cdot \overline a_t}_{U^s(\Z_N)}^{2^s}\Bigr)^{1/2^{s+1}}.
$$
If $f$ is a function on $\N$, then $\norm {f_N}_{U^s(\Z_N)}$ is defined by considering the function $f_N=f\cdot \one_{[N]}$ as a function on $\Z_N$.
\end{definition}

The following structural result from \cite{FH14} is crucial for our study:

\begin{theorem}[Structure theorem for multiplicative functions~{\bf \cite[Theorem~8.1]{FH14}}]
\label{T:Structure}
Let $s\in \N$  and  $\ve>0$. Then there exists a  real number  $\theta>0$ and $N_0\in \N$, depending on $s$ and $\ve$ only,    such that
 for every  prime $N\geq N_0$,    every $f\in\CM_\C$ admits
the decomposition
$$
 f(n)=f_{N,\st}(n)+f_{N,\un}(n), \quad \text{ for every }\  n\in [N],
$$
 where  $f_{N,\st},f_{N,\un}\colon [N]\to \C$ are bounded by $1$ and $2$ respectively and   satisfy:
\begin{enumerate}
\item
\label{it:weakUs-1}  $f_{N,\st}=f_N*\phi_{N,\theta}$ where
$\phi_{N,\theta}$ is the  kernel on $\Z_N$ defined  previously and
  the convolution product is defined  in $\Z_N$;
\item
\label{it:weakUs-3}
 $\norm{f_{N,\un}}_{U^s(\Z_N)}\leq\ve$.
\end{enumerate}
\end{theorem}
We think of $f_{N,\st}$ and  $f_{N,\un}$ as the structured and uniform component of $f$ respectively.

From this point on we assume that
$N>2QV$. When convenient we identify
$\Z_N$ with the set $\{0,\ldots, N-1\}$
and we denote by  $(a,b) \! \! \! \mod{N}$
the set  that consists of those  $\xi\in  \Z_N$
such that $\xi+kN\in (a,b)$ for some $k\in \Z$.
 Note that $\xi\in \Xi_{N,\theta}$ if and only if there exists  $p\in \Z$
such that $\xi-\frac{p}{Q}N\in (-V,V) \! \! \mod{N}$. Hence,
$$
\Xi_{N,\theta}=\bigcup_{p=0}^{Q-1}\big(\frac{p}{Q}N-V,  \frac{p}{Q}N+V\big) \! \! \! \mod{N}.
$$
We may choose to include or omit the endpoints of each interval (if they are integers), since for these values
the Fourier transform of the kernel is $0$. Hence, we can assume that
\begin{equation}\label{E:StructureXi1}
\Xi_{N,\theta}=\bigcup_{p=0}^{Q-1} \Xi_{N,\theta,p}
\end{equation}
where for $p=0,\ldots, Q-1$ we have
$
\Xi_{N,\theta,p}:=\big\{\big\lfloor\frac{p}{Q}N\big\rfloor+j \bmod{N}\colon -V<j\leq V\big\}.
$
Note that for fixed $N>2QV$ and $\theta>0$ the sets $\Xi_{N,\theta,p}$, $p=0,\ldots, Q-1$, are disjoint, each of cardinality $2V$,
hence $|\Xi_{N,\theta}|=2QV$.
Furthermore,  if $N\equiv 1 \bmod{Q}$,   then
\begin{equation}\label{E:StructureXi2}
 \Xi_{N,\theta,p}=\big\{\frac{p}{Q}(N-1)+j  \bmod{N}\colon -V<j\leq V\big\}.
\end{equation}
Restricting $N$ to a specific congruence class $\bmod Q$ is needed in the proof of Lemma~\ref{L:Structured'}.

\section{Proof of the main results}\label{S:proofs}
\subsection{Preparatory lemmas} In what follows we  use repeatedly the following simple fact:
If $(w(n))$ is a slowly-varying sequence, then for every complex valued bounded sequence $(a(n))$ we have
$$
\frac{1}{N}\sum_{n=1}^Na(n)\e(w(n))=\e(w(N))\cdot \frac{1}{N}\sum_{n=1}^Na(n)+o_{N\to \infty}(1).
$$

We start with some preliminary lemmas.

 \begin{lemma}\label{L:Partial'}
Let $(a(n))$
be a bounded sequence of complex numbers and $(w(n))$ be a slowly-varying sequence. Suppose that there exist $c\in \C$ and $t\in \R$ such that
\begin{equation}\label{E:assumption}
\frac{1}{N}\sum_{n=1}^Na(n)=cN^{it}\e(w(N))+o_{N\to \infty}(1).
\end{equation}
Then  for every $\alpha\in \R$  we have
$$
\frac{1}{N}\sum_{n=1}^Na(n)\e\big( n\frac{\alpha}{N}\big)=c'N^{it}\e(w(N))+o_{N\to \infty}(1)
$$
where $c':=c (1+it)  \int_{0}^1 y^{it}\e\big(  \alpha y\big)\, dy$.
\end{lemma}

\begin{proof}
Without loss of generality we can assume that $c=1/(1+it)$.

We first claim that
\begin{equation}\label{E:a1}
\frac{1}{N}\sum_{n=1}^Na(n)\e\big( n\frac{\alpha}{N}\big)=\e(w(N))\cdot \frac{1}{N}\sum_{n=1}^Nn^{it} \e\big( n\frac{\alpha}{N}\big)+o_{N\to \infty}(1).
\end{equation}
Indeed,
for $n\in \N$ let
 $$
 S(n):=\sum_{k=1}^n\big(a(k)-k^{it}\e(w(k))\big).
  $$
  Since  $w$ is a slowly-varying  sequence and  $\frac{1}{N}\sum_{n=1}^Nn^{it}=\frac{N^{it}}{1+it}+o_{N\to\infty}(1)$, we get that $\frac{1}{N}\sum_{n=1}^Nn^{it}\e(w(n))=\frac{N^{it}}{1+it}\e(w(N))+o_{N\to\infty}(1)$; hence \eqref{E:assumption} gives that   $S(n)/n\to 0$ as $n\to \infty$.
Using partial summation we see that the modulus of the average
  $$
\frac{1}{N}\sum_{n=1}^N\big(a(n)-n^{it}\e(w(n))\big)\e\big( n\frac{\alpha}{N}\big)
  $$
 is at most
$$
\frac{1}{N}\big(\sum_{n=2}^{N-1}|S(n)|\big|\e\big( (n+1)\frac{\alpha}{N}\big)-\e\big( n\frac{\alpha}{N}\big)\big|+|S(N)|\big)+o_{N\to\infty}(1).
$$
Let $\varepsilon>0$. Since  $S(n)/n\to 0$ as $n\to \infty$ we have $|S(n)|/n\leq \varepsilon$
for every sufficiently large $n$, and thus the last expression  is bounded by
$$
\frac{1}{N}\big(\sum_{n=2}^{N-1}\varepsilon n\frac{|2\pi\alpha|}{N}+\varepsilon N\big)+o_{N\to\infty}(1)\leq (|\pi\alpha|+1)\varepsilon +o_{N\to\infty}(1).
$$
Since $\varepsilon$ is arbitrary we get that
$$
\frac{1}{N}\sum_{n=1}^N\big(a(n)-n^{it}\e(w(n))\big)\e\big( n\frac{\alpha}{N}\big)=o_{N\to\infty}(1)
$$
and the asymptotic \eqref{E:a1}  follows because $w$ is  slowly-varying.

Lastly, note that
$$
\frac{1}{N}\sum_{n=1}^Nn^{it}\e\big( n\frac{\alpha}{N}\big)=
N^{it}
\cdot  \frac{1}{N}\sum_{n=1}^N\big(\frac{n}{N}\big)^{it}\e\big( n\frac{\alpha}{N}\big).
$$
Interpreting the last average as a Riemann sum we get that it converges to the integral
$\int_0^1 y^{it}\e(\alpha y)\, dy$ (integration by parts shows that the integral converges).
Hence,
$$
\frac{1}{N}\sum_{n=1}^Nn^{it}\e\big( n\frac{\alpha}{N}\big)=c'\cdot N^{it}+o_{N\to\infty}(1)
$$
where $c':=\int_{0}^1 y^{it}\e\big( \alpha y\big)\, dy$.

 Combining the above we get the asserted claim.
 \end{proof}
The next lemma enables us to get asymptotics for the discrete Fourier transform of elements of $\CM_\C$ for certain ``major arc'' frequencies.
\begin{lemma}
\label{L:Fourier'}
 Let $f\in \CM_\C$,  $t=t_f$, $w=w_f$,  be as in Theorem~\ref{T:HalaszAP}.
 Furthermore, let  $Q\in \N$, $p,\xi'\in \Z$,
  and
 $$\xi_N=\frac{p}{Q}N+\frac{\xi'}{Q}, \quad N\in \N.$$
  Then there exists a constant $c=c_{f,p,Q,\xi'}\in \C$ such that
\begin{equation}
\label{E:fouri}
 \frac{1}{N}\sum_{n=1}^{N}f(n)\e\big(-n\frac{\xi_N}{N}\big)=cN^{it}\e(w(N))+o_{N\to\infty}(1).
\end{equation}
\end{lemma}

 \begin{remark}
 We are going to apply this for integers $p,\xi', N$ such that $N\equiv 1 \bmod{Q}$ and  $p+\xi'\equiv 0 \bmod{Q}$, in which case $\xi_N$ is an integer.
\end{remark}

\begin{proof}
 Notice first that the left hand side in \eqref{E:fouri} is equal to
 $$
 \frac{1}{Q}\sum_{r=1}^Q\e\big(-r\frac{p}{Q}\big) \frac{1}{\lfloor N/Q \rfloor}\sum_{n=1}^{ \lfloor N/Q \rfloor}f(Qn+r)\e\big(-(Qn+r)\frac{\xi'}{QN}\big)+o_{N\to\infty}(1).
 $$
 Hence,   it suffices to show that for every fixed $Q,\xi',$ and $r\in [Q]$, we have the asserted asymptotic for
the averages
$$
 \frac{1}{ \lfloor N/Q \rfloor}\sum_{n=1}^{\lfloor N/Q\rfloor}f(Qn+r)\e\big(-(Qn+r)\frac{\xi'}{QN}\big).
 $$
 Since $\e(-r\xi'/(QN))\to 1$ as $N\to \infty$ it suffices to prove the asserted asymptotic
for the averages
 \begin{equation}\label{E:Av1'}
 \frac{1}{ \lfloor N/Q \rfloor}\sum_{n=1}^{ \lfloor N/Q \rfloor}f(Qn+r)\e\big(-n\frac{\xi'}{N}\big).
 \end{equation}
  By Theorem~\ref{T:HalaszAP} there exists $c=c_{f,Q,r}\in \C$ such that
 $$
\frac{1}{N}\sum_{n=1}^Nf(Qn+r)=cN^{it}\e(w(N))+o_{N\to\infty}(1).
 $$
Using Lemma~\ref{L:Partial'} for $a(n):=f(Qn+r)$  and that $\lfloor N/Q \rfloor^{it}-(N/Q)^{it}\to 0$ and  $w(N)-w(\lfloor N/Q \rfloor)\to 0$ as $N\to \infty$ (since $w$ is slowly-varying), we deduce  the needed asymptotic for the averages  \eqref{E:Av1'}. This completes the proof.
 \end{proof}
Next we   analyze the asymptotic behavior of the averages \eqref{E:multi1} when in place of the multiplicative functions $f_1,\ldots, f_\ell$ we use  their structured components given by Theorem~\ref{T:Structure}.
\begin{lemma}
\label{L:Structured'}
  Let  $\theta>0$, $d,Q\in \N$,  $f_1,\ldots, f_\ell \in \CM_\C$, and
  $t_{f_j}$, $w_{f_j}$, $j=1,\ldots, \ell$,   be as in Theorem~\ref{T:HalaszAP}. Let also  $L_1,\ldots, L_\ell\colon \N^d\to \N$ be linear forms  let $\kappa$ be twice the sum of their coefficients.
For $N\in \N$ let $\tN>N$  be a prime that satisfies  $\tN\equiv 1  \bmod{Q}$ and suppose that the limit $\beta:=\lim_{N\to \infty} (N/\tN)$ exists  and is smaller  than or equal to $\kappa\inv$.
For $j=1,\ldots, \ell,$ let
$f_{j,\tN,\st}:=f_{j,\tN}*\phi_{\tN,\theta}$ where $\phi_{\tN,\theta}$ is defined by \eqref{eq:fourier-phi} and the convolution product is defined in $\Z_\tN$.
Then there exists  $c\in \C$ such that
 \begin{equation}\label{E:Str1'}
\frac{1}{N^d}\sum_{\bm \in  [N]^d} \prod_{j=1}^\ell f_{j,\tN, \st}(L_j(\bm))=c  N^{it}\e(w(N))+o_{N\to\infty}(1)
\end{equation}
where $t:=\sum_{j=1}^\ell t_{f_j}$ and $w:=\sum_{j=1}^\ell w_{f_j}$.
\end{lemma}
\begin{proof}
By the definition of $f_{j,\tN,\st}$  we have for $j=1,\ldots,\ell$ that
$$
f_{j,\tN,\st}(n)= \sum_{\xi \in \Xi_{\tN,\theta}} \widehat{f_{j,\tN}}(\xi)\, \widehat{\phi_{\tN,\theta}}(\xi) \, \e\big(n\frac{\xi}{\tN}\big), \quad n\in
[\tN],
$$
where  $\Xi_{\tN,\theta}$ is the spectrum of $\phi_{\tN,\theta}$ (defined in  \eqref{eq:def-Xi}).

Recall that for $1\leq j\leq\ell$, the linear form $L_j$ has non-negative integer coefficients. By hypothesis, for every $\bm\in[N]^d$ we have $0\leq L_j(\bm)\leq\kappa N/2$.
On the other hand,
for $N$ large enough we have $\wt N\geq \kappa N/2$ and thus,  for  every $\bm\in[N]^d$,
we have $L_j(\bm)\in[\wt N]$. Therefore, the last formula holds for  $n=L_j(\bm)$.

Since $\tN\equiv 1  \bmod{Q}$ it follows from  \eqref{E:StructureXi1} and \eqref{E:StructureXi2}  that for $\tN>2QV$ if   $\xi\in \Xi_{\tN,\theta}$, then $\xi$ can be uniquely represented as
$$
\xi=\frac{p}{Q}\tN+\frac{\xi'}{Q}
$$
for some $p\in \{0,\ldots, Q-1\}$ and $\xi'\in \Xi'_{p,\theta}$ where for $p=0,1,\ldots, Q-1$ we have
$$
\Xi'_{p,\theta}:=
\big\{-p+jQ\colon -V< j\leq V\big\}.
$$
Hence, it suffices to show that the averages \eqref{E:Str1'} satisfy the asserted asymptotic
  when
for $j=1,\ldots, \ell$ the (finite) sequence  $(f_{j,\tN,\st}(n))_{n\in [\tN]}$ in \eqref{E:Str1'} is replaced by
the sequence
$$
\widehat{f_{j,\tN}}\big(\frac{p_j}{Q}\tN+\frac{\xi_j'}{Q}\big)\cdot \widehat{\phi_{\tN,\theta}}
\big(\frac{p_j}{Q}\tN+\frac{\xi_j'}{Q}\big)\cdot  \e\big(n(\frac{p_j}{Q}+\frac{\xi_j'}{Q\tN})\big), \quad n\in[\tN],
$$
for all possible vectors $(p_1,\ldots, p_\ell)$, $(\xi_1',\ldots, \xi_\ell')$, where $p_j\in \{0,\ldots, Q-1\}$ and  $\xi_j'\in \Xi'_{p_j,\theta}$  for  $j=1,\ldots, \ell$.

For $j=1,\ldots, \ell,$ by Lemma~\ref{L:Fourier'}  we have  that there exists $c_j=c_{f,p_j,\theta, \xi_j'}\in \C$ such that
$$
 \widehat{f_{j,\tN}}\big(\frac{p_j}{Q}\tN+\frac{\xi_j'}{Q}\big)=c_j \tN^{it_{f_j}}\e(w_{f_j}(\tN))+o_{N\to\infty}(1)=c'_j N^{it_{f_j}}\e(w_{f_j}(N))+o_{N\to\infty}(1)
$$
where $c'_j:=\beta^{-it_{f_j}}c_j$ and the last identity follows because  $\lim_{N\to \infty}  \frac{N}{\tN}=\beta$ and $w_{f_j}$ is a slowly-varying
sequence.
Hence, there exists $c=c_{f,p_1,\ldots, p_\ell,\beta,\theta, \xi_1',\ldots, \xi_\ell'}\in \C$ such that
$$
 \prod_{j=1}^\ell\widehat{f_{j,\tN}}\big(\frac{p_j}{Q}\tN+\frac{\xi_j'}{Q}\big)=c N^{it}\e(w(N))+o_{N\to\infty}(1)
$$
where $t:=\sum_{j=1}^\ell t_{f_j}$ and $w:=\sum_{j=1}^\ell w_{f_j}$.

Furthermore,   it follows from  \eqref{eq:fourier-phi}  that
$$
\widehat{\phi_{\tN,\theta}}
\big(\frac{p_j}{Q}\tN+\frac{\xi_j'}{Q}\big)=1-\frac{\xi_j'}{QV}, \quad \text{whenever }\  \tN\geq 2QV.
$$

Finally, we deal with the terms  $\e\big(n(\frac{p_j}{Q}+\frac{\xi_j'}{Q\tN})\big)$.
After  substituting $L_1(\bm),\dots,L_\ell(\bm)$ for $n$, and writing $\bm=(m_1,\ldots,m_d)$,  they give  rise to averages of the form
$$
\frac{1}{N^d}\sum_{1\leq m_1,\ldots,m_d\leq N} \e\Big(\sum_{j=1}^d m_j\big(\frac{k_j}{Q}+\frac{l_j}{Q\tN}\big)\Big)
$$
for some $k_j,l_j \in \Z$, $j=1,\ldots,d,$ which depend only on the coefficients of the linear forms and the set $\bigcup_{p=0}^{Q-1}\Xi'_{p,\theta}$. If
for some $j\in \{1,\ldots, \ell\}$ the integer $k_j$ is not divisible by $Q$,  then this average converges to $0$. Otherwise,  it converges to $\prod_{j=1}^d\int_0^1\e\big(l_j\beta s/Q)\, ds$.

 Combining the above we get the asserted asymptotic \eqref{E:Str1'}.
\end{proof}

  Next we state a variant of some  uniformity estimates that appear in~\cite[Proposition~7.1]{GT10}. They can be obtained using an argument similar to the one in~\cite[Lemma~9.6]{FH14}; we present it for completeness.

\begin{lemma}[Uniformity estimates]
\label{L:seminorms}
Let  $d,\ell\in \N$, with $\ell \geq 3$,  and $L_j\colon \N^d\to \N$, $j=1,\ldots, \ell$,  be  linear forms such that   the  forms $L_1,L_j$ are independent for $j=2,\ldots,  \ell$.
Let $\kappa$ be twice the sum of  the coefficients of the forms $L_j$
and $K\in \N$ with $K>\kappa$.
 For $N\in \N$ large enough, let  $\tN$ be a prime with  $\kappa N\leq \tN\leq KN$  and $a_{N,1}\ldots,a_{N,\ell}\colon [\tN]\to \C$ be  arbitrary sequences  bounded by $1$.  Then there exist  positive constants
 $c=c(d,\ell)$ and $C=C(d,\ell,K)$, such that
\begin{equation}\label{E:uniform}
\Big|\frac{1}{N^d}\sum_{\bm\in [N]^d}
\prod_{j=1}^\ell a_{N,j}(L_j(\bm))\Big|\leq C\, \norm{a_{N,1}}_{U^{\ell-1}(\Z_\tN)}^c+o_N(1).
\end{equation}
\end{lemma}

\begin{proof}
After putting together terms evaluated at linear forms $L_j$, $j=2,\ldots, \ell$, that are pairwise dependent we can assume that the linear forms are pairwise independent.

 For $j=1,\ldots, \ell$ and $N\in \N$, let $\wt a_{N,j}\colon \Z \to \C$  be periodic of period $\wt N$ and
   equal to  $a_{N,j}$ on the interval
 $\bigl[-\lceil \wt N/2\rceil,   \lfloor \wt N/2\rfloor\bigr)$.
For $\bm\in [N]^d$, since $|L_j(\bm)|<\wt N/2$,  we have $a_{N,j}(L_j(\bm))=\wt a_{N,j}(L_j(\bm))$. Hence,
\begin{equation}\label{E:uniform1}
 \frac{1}{N^d}\sum_{\bm\in [N]^d}  \prod_{j=1}^\ell a_{N,j}(L_j(\bm))
 =\big(\frac{\tilde{N}}{N}\big)^d\cdot
 \frac{1}{\tN^d}\sum_{\bm\in \Z_{\tN}^d}\one_{[N]^d}(\bm)\, \prod_{j=1}^\ell \wt a_{N,j}(L_j(\bm)).
\end{equation}
Henceforth, we work with the right hand side and assume that the linear forms $L_j$ and the functions $\wt a_{N,j}$  are defined on $\Z_\tN$.

Our first goal is to remove the cut-off ${\bf 1}_{[N]^d}(\bm)$.
 To do this, one can follow the method in~\cite[Proposition~7.1]{GT10},
 or what turns out to be  somewhat  simpler,  follow the argument in~\cite[Lemma~A.1]{FH14};
 after approximating the cut-off by a product of smoothed out cut-offs in $\Z_\tN$ we deduce  that for some $C'=C'(d,K)$ the modulus of the right hand side in \eqref{E:uniform1}
 is bounded by
 \begin{equation}\label{E:uniform1'}
C'\cdot \max_{{\bf \xi}\in\Z_{\tN}^d}
\Bigl|\frac{1}{\tN^d}\sum_{\bm\in \Z_{\tN}^d} \e\big(\frac{\bm\cdot{\bf \xi}}{\wt  N}\big) \prod_{j=1}^\ell\wt a_{N,j}(L_j(\bm)) \Bigr|^{\frac{1}{d+1}}+o_{N\to \infty}(1).
\end{equation}

 Next we  estimate the averages in \eqref{E:uniform1'}.
The pairwise independence of the linear forms $L_j$
implies that for $\tN \geq \kappa^2$ the forms $L_j$ on $\Z_\tN^d$ are pairwise independent over $\Z_\tN$.
Using this and an    iteration of the Cauchy-Schwarz inequality (see for example the argument in~\cite[Proposition~7.1]{GT10}) we get
\begin{equation}\label{E:uniform2}
 \max_{{\bf \xi}\in\Z_{\tN}^d}
\Bigl|\frac{1}{\tN^d}\sum_{\bm\in \Z_{\tN}^d} \e\big(\frac{\bm\cdot{\bf \xi}}{\wt  N}\big) \prod_{j=1}^\ell\wt a_{N,j}(L_j(\bm)) \Bigr|\leq
\norm{\wt a_{N,1}}_{U^{\ell-1}(\Z_{\tN})}.
 \end{equation}
 Note that the exponential terms are going to vanish in the process because $\ell\geq 3$.

Finally, we write
$$
\Z_\tN=I_N\cup J_N\cup\{0\},\ \text{ where }\ I_N:= [1,\lfloor \wt N/2\rfloor)\ \text{ and }\
J_N:= [\lfloor \tN/2\rfloor, \tN).
$$
Note that  $\norm{\one_{\{0\}}\cdot \wt a_{N,1}}_{U^{\ell-1}(\Z_\tN)}\to 0$ as $N\to+\infty$.
Furthermore, by the proof  of \cite[Lemma~A.1]{FH14} we have
\begin{equation}
\label{E:uniform3}
\max_{I\subset \Z_\tN}\norm{\one_{I_N}\cdot a_{N,1}}_{U^{\ell-1}(\Z_\tN)}\leq
3 \norm{a_{N,1}}_{U^{\ell-1}(\Z_\tN)}^{1/(2^{\ell-1}+1)}+o_N(1)
\end{equation}
where the maximum is taken over all subintervals $I$ of $\Z_\tN$.
Since $\wt a_{N,1}$ and $a_{N,1}$ coincide on $I_N$, we have $\norm{\one_{I_N}\cdot \wt a_{N,1}}_{U^{\ell-1}(\Z_\tN)}=\norm{\one_{I_N} \cdot a_{N,1}}_{U^{\ell-1}(\Z_\tN)}$.
For $n\in J_N$ we have $\wt a_{N,1}(n)=\wt a_{N,1}(n-\tN)=a_{N,1}(n-\tN)=a_{N,1}(\tN-n)$. The map $n\mapsto \tN-n$ maps the interval $J_N$ onto the interval $J_N':=[1,\lceil \wt N/2\rceil]$ and thus $\norm{\one_{J_N}\cdot \wt a_{N,1}}_{U^{\ell-1}(\Z_\tN)}=
 \norm{\one_{J_N'}\cdot  a_{N,1}}_{U^{\ell-1}(\Z_\tN)}$.  The asserted estimate now follows by combining \eqref{E:uniform1}-\eqref{E:uniform3}.
\end{proof}

\subsection{Proof of the main results}
\label{subsec:proof-of-main}
We proceed to prove the main results of this article.
\begin{proof}[Proof of Theorem~\ref{T:Main}]
 Without loss of generality we can  assume that $\ell\geq 3$.
   After putting together terms evaluated at linear forms that are pairwise dependent we can assume that the linear forms are pairwise independent.

Let  $\varepsilon>0$.
We let $\kappa$ be twice  the sum of  the coefficients of the forms $L_j$ and
$$
 \delta:=\big(\frac{\varepsilon}{2C\ell}\big)^{1/c}
$$
where $c,C$ are as in Lemma~\ref{L:seminorms}.
We use the structural result of  Theorem~\ref{T:Structure} for this $\delta$ in place of $\varepsilon$ and for $\ell-1$ in place of $s$.
We  get that  there exists $\theta=\theta(\ve,\ell)>0$ such that for all large enough $N\in \N$, if  $\tN$ denotes  the smallest prime such that $\tN>\kappa N$ and
$\tN\equiv 1  \bmod{Q}$ ($Q$ was introduced in Section~\ref{S:2} and  depends only on $\theta$),
then for $j=1,\ldots, \ell$, we have the  decompositions
\begin{equation}\label{E:decomposition}
  f_j(n)=f_{j,\tN,\st}(n)+ f_{j,\tN,\un}(n), \qquad n\in [\tN],
\end{equation}
where   $f_{j,\tN,\st}=f_{j,\tN}*\phi_{\tN,\theta}$ ($\phi_{\tN,\theta}$ is defined by \eqref{eq:fourier-phi}) and
\begin{equation}\label{E:bound}
\norm{f_{j,\tN,\un}}_{U^{\ell-1}(\Z_\tN)}\leq\delta.
\end{equation}
The prime number theorem on arithmetic progressions implies  that
$$
\lim_{N\to \infty}\frac{N}{\tN}=\frac{1}{\kappa}.
$$
We remark that the hypothesis of Lemma~\ref{L:Structured'} are satisfied. As in the proof of this lemma, for $N$ sufficiently large we have $L_j(\bm)\in[\wt N]$ for $j=1,\dots,\ell$ and every   $\bm\in [N]^d$ and thus  equation \eqref{E:decomposition} applies for  $L_j(\bm)$ in place of $n$.

For $N\in \N$, given
  $a_{N, 1},\ldots, a_{N,\ell}\colon [\tN]\to \C$
 we define
$$
A_N(a_{N,1},\ldots, a_{N,\ell}):=\frac{1}{N^d}\sum_{\bm\in [N]^d} \prod_{j=1}^\ell a_{N,j}(L_j(\bm)).
$$
Since for  $j=1,\ldots, \ell,$ the functions $f_{j,\tN,\un}\colon [\tN] \to \C$  are   bounded by $2$, it follows from Lemma~\ref{L:seminorms} and \eqref{E:bound} that
$$
|A_N(a_{N, 1},\ldots, a_{N,\ell})|\leq \ve/\ell+o_N(1)
$$
 if $a_{N,j}=f_{j,\tN,\un}$ for some $j=1,\ldots, \ell$ and all other sequences $a_{N,j}$ are  bounded by $1$.
Using this property, equation \eqref{E:decomposition},  the fact that $f_j, f_{j,\tN,\st}$ are bounded by $1$ for $j=1,\ldots, \ell$,   and telescoping, we deduce that
$$
\limsup_{N\to\infty}|A_{N}(f_1,\ldots, f_\ell)-A_{N}(f_{1,\tN,\st},\ldots, f_{\ell, \tN,\st})|\leq \varepsilon.
$$
Furthermore,   by Lemma~\ref{L:Structured'} we have that the   limit
$$
\lim_{N\to \infty} N^{-it}\e(-w(N))  A_{N}(f_{1,\tN,\st},\ldots, f_{\ell, \tN,\st})
$$
exists for  $t:=\sum_{i=1}^\ell t_{f_i}$ and $w:=\sum_{i=1}^\ell w_{f_i}$.
It follows that
\begin{multline*}
 \limsup_{N\to \infty}\Re\big(N^{-it}\e(-w(N))A_N(f_1,\ldots, f_\ell)\big)\leq \\ \liminf_{N\to \infty}\Re\big(N^{-it}\e(-w(N))A_N(f_1,\ldots, f_\ell)\big)+ 2\varepsilon
\end{multline*}
 and a similar  estimate holds for the imaginary parts.
 Since $\varepsilon$ is arbitrary and all expressions are bounded, the limit $\lim_{N\to \infty} \big(N^{-it}\e(-w(N))A_N(f_1,\ldots, f_\ell)\big)$ exists. This proves the asserted asymptotic.

 Finally, we prove the last claim of Theorem~\ref{T:Main}. If the linear forms are pairwise independent and one of the multiplicative functions is non-pretentious,
then, using the terminology of \cite{FH14}, it is  aperiodic (see \cite[Proposition~2.3]{FH14}). Hence,   by \cite[Theorem~9.7]{FH14} the averages \eqref{E:multi1} converge to $0$ as $N\to \infty$.
\end{proof}

\begin{proof}[Proof of Theorem~\ref{T:1}]
After putting together terms evaluated at linear forms that are pairwise dependent we can assume that the linear forms are pairwise independent. Since the multiplicative functions take real values, it follows by the
definition of $t_f$ and $w_f$ in Theorem~\ref{T:HalaszAP} (see  the remark following this  theorem)
 that   $t_{f_j}=0$ and $w_{f_j}=0$ for $j=1,\ldots, \ell$.
The result  is now immediate  from Theorem~\ref{T:Main}.
\end{proof}
\begin{proof}[Proof of Theorem~\ref{T:2}]
Since $|N^{it}\e(w(N))|=1$ for every $N\in \N$  the result follows  immediately  from Theorem~\ref{T:Main}.
\end{proof}

\begin{proof}[Proof of Theorem~\ref{T:3}]
  It follows from the definition of $t_f$ and $w_f$ in Theorem~\ref{T:HalaszAP} that if $f\in \CM_\C$, then $t_{\bar{f}}=-t_f$ and   $w_{\bar{f}}=-w_f$.  The result then follows  immediately  from Theorem~\ref{T:Main}.
\end{proof}

\begin{proof}[Proof of Theorem~\ref{T:4}]
In what follows  we denote by $TF$ the composition $F\circ T$.

The  polarization identity
\begin{multline*}
4\langle T_nF, T_mG \rangle=\langle T_n(F+G), T_m(F+G) \rangle
-\langle T_n(F-G), T_m(F-G) \rangle+\\
i\langle T_n(F+iG), T_m(F+iG) \rangle-i \langle T_n(F-iG), T_m(F-iG) \rangle
\end{multline*}
 allows us to  express a
sequence of the form $\int T_nF\cdot T_mG \ d\mu$, $F,G\in L^2(\mu)$,  as  a linear combination of
sequences of the form $\int T_nH\cdot T_m\bar{H} \ d\mu$ with $H\in L^2(\mu)$. It thus suffices to check the asserted
convergence of  ergodic averages when $G=\bar{F}$.

 To this end, note  that  the action $(T_n)_{n\in \N}$ on  $(X,\CX,\mu)$ extends to a measure preserving action $(T_r)_{r\in\Q^+}$ of the multiplicative group $\Q^+$
by defining
$$
T_{a/b}:=T_aT_b\inv\ \ \text{ for all }\ \  a,b\in\N.
$$
Let $\rho\colon\Q^+\to\C$ be defined by
$$
\rho(r):=\int T_rF\cdot \overline{F}\, d\mu, \quad r\in \Q^+.
$$
 Then   $\rho$ is positive definite on $(\Q^+,\cdot$), that is, for every $n\in\N$, all $r_1,\dots,r_n\in\Q^+$, and all $\lambda_1,\dots,\lambda_n\in\C$, we have
$$
\sum_{i,j=1}^n\lambda_i\overline{\lambda_j}\,\rho(r_i\,r_j\inv)\geq 0.
$$
By Bochner's theorem, there exists a unique  positive finite measure  $\nu$ on the
 dual group of the
 group $\Q^+$ with multiplication,  with a  Fourier-Stieltjes transform $\wh \nu$ equal to the function $\rho$.
The dual group of the
multiplicative group $\Q^+$ is the space
$\CM_{\T}^c$  of all completely multiplicative functions of modulus $1$,  the duality being given by
$$
f(m/n)=f(m)\overline f(n) \ \ \text{ for every }\ \ f\in \CM_{\T}^c
\ \text{ and every }\ m,n\in\N.
$$
The group $\CM_{\T}^c$ is endowed with the dual topology, which is  simply the compact topology of pointwise convergence.

 It follows from the previous discussion that for every function $F\in L^2(\mu)$ there
exists a positive finite
  measure $\nu$ on the compact Abelian group $\CM_{\T}^c$,  such that, for all $m,n\in\N$,
$$
\int T_mF\cdot \overline T_nF\,d\mu
=\int T_{m/n}F\cdot \overline F\,d\mu=\int_{\CM_{\T}^c}f(m/n)\,d\nu(f)
=\int_{\CM_{\T}^c}f(m)\, \overline f(n)\,d\nu(f).
$$
Hence, in order to show convergence  of the averages \eqref{E:ergodic} it suffices to prove that the following averages converge
$$
\frac{1}{N^d}\sum_{\bm\in [N]^d} \int_{\CM_{\T}^c}f(\prod_{j=1}^\ell L_j(\bm))\cdot  \overline f(\prod_{j=1}^\ell L'_j(\bm))\,d\nu(f)
 $$
 as $N\to \infty$.
Since $f$ is completely multiplicative  this follows from  Theorem~\ref{T:3} and the bounded convergence theorem.
\end{proof}

\end{document}